\documentclass[12pt]{article}

\usepackage{verbatim}

\usepackage{amsthm}
\usepackage{amsfonts}
\usepackage{amsmath}
\usepackage{amssymb}
\usepackage{color,enumerate}

\usepackage{multicol,url}

\begin{document}
 \newcounter{thlistctr}
 \newenvironment{thlist}{\
 \begin{list}%
 {\alph{thlistctr}}%
 {\setlength{\labelwidth}{2ex}%
 \setlength{\labelsep}{1ex}%
 \setlength{\leftmargin}{6ex}%
 \renewcommand{\makelabel}[1]{\makebox[\labelwidth][r]{\rm (##1)}}%
 \usecounter{thlistctr}}}%
 {\end{list}}

\thispagestyle{empty}

\newtheorem{Lemma}{\bf LEMMA}[section]
\newtheorem{Theorem}[Lemma]{\bf THEOREM}
\newtheorem{Claim}[Lemma]{\bf CLAIM}
\newtheorem{Corollary}[Lemma]{\bf COROLLARY}
\newtheorem{Proposition}[Lemma]{\bf PROPOSITION}
\newtheorem{Example}[Lemma]{\bf EXAMPLE}
\newtheorem{Fact}[Lemma]{\bf FACT}
\newtheorem{definition}[Lemma]{\bf DEFINITION}
\newtheorem{Notation}[Lemma]{\bf NOTATION}
\newtheorem{remark}[Lemma]{\bf REMARK}

\title{Implication Zroupoids and Identities of Associative Type}
             
\author{Juan M. Cornejo\footnote{I wish to dedicate this work to the memory of Oscar Foresi, my second dad.} \ and Hanamantagouda P. Sankappanavar} 

\date{}

\maketitle

\numberwithin{equation}{section}

\thispagestyle{empty}

\begin{abstract}

An algebra $\mathbf A = \langle A, \to, 0 \rangle$, where $\to$ is binary and $0$ is a constant, is called 
	 an {\it implication zroupoid} ($\mathcal I$-zroupoid, for short) if $\mathbf A$ satisfies the identities:	
	 	$(x \to y) \to z \approx [(z' \to x) \to (y \to z)']'$  
	 and   $ 0'' \approx 0$, where $x' : = x \to 0$, and $\mathcal I$ denotes the variety of all  
$\mathcal I$-zroupoids.   An $\mathcal I$-zroupoid is {\it symmetric} if it satisfies $x'' \approx x$ and $(x \to y')' \approx (y \to x')'$.  The variety of symmetric $\mathcal I$-zroupoids is denoted by $\mathcal S$.  	 
An identity $p \approx q$, in the groupoid language $\langle \to \rangle$, is called {\it an identity of associative type of length $3$} if 
 $p$ and $q$ have exactly 3 (distinct) variables, say x,y,z, and 
 are grouped according to one of the 
two ways of grouping:
(1) $\star \to (\star \to \star)$ and
(2)  $(\star \to \star) \to \star$, where $\star$ is a place holder for a variable.
A subvariety of $\mathcal I$ is said to be of {\it associative type of length $3$,} if it is defined, relative to  
$\mathcal I$, by 
a single identity of associative type of length $3$.

In this paper we give a complete analysis of the mutual relationships of all subvarieties of $\mathcal I$ 
of associative type of length $3$.  We prove, in our main theorem, that there are exactly 8 such subvarieties of $\mathcal I$ that are
distinct from each other and describe explicitly the poset formed by them
under inclusion.
As an application of the main theorem,  we  derive that there are three distinct subvarieties of the variety $\mathcal S$, each defined, relative to $\mathcal S$, by a single identity 
of  associative type of length $3$. 
\end{abstract}

\section{\bf Introduction}

In 1934, Bernstein gave a system of axioms for Boolean algebras in \cite{bernstein1934postulates} using implication alone.  
Even though his system was not equational, it is not hard to see that one could easily convert it into an equational one by using an additional constant.  
In 2012, the second author extended this ``modified Bernstein's theorem''
to De Morgan algebras in \cite{sankappanavarMorgan2012} by  
 showing that the variety of De Morgan algebras, 
 is term-equivalent  
 to the variety $\mathcal{DM}$ 
 (defined below)  
whose defining axioms use only an implication $\to$ and a constant $0$.   

The primary role played by the identity (I): $(x \to y) \to z \approx [(z' \to x) \to (y \to z)']'$, where $x' : = x \to 0$, 
in the axiomatization of each of those new varieties 
motivated the second author to study the identity (I) in its own right and led him to 
introduce a new equational class of algebras called ``implication zroupoids'' in \cite {sankappanavarMorgan2012} (also called ``implicator groupoids'' in \cite{cornejo2015implication}).   

	An algebra $\mathbf A = \langle A, \to, 0 \rangle$, where $\to$ is binary and $0$ is a constant, is called a {\rm zroupoid}.
	A zroupoid $\mathbf A = \langle A, \to, 0 \rangle$ is an {\rm implication zroupoid} \rm ($\mathcal I$-zroupoid, for short\rm) if $\mathbf A$ satisfies:
	\begin{itemize}
		\item[\rm{(I)}] 	$(x \to y) \to z \approx [(z' \to x) \to (y \to z)']'$, where $x' : = x \to 0$,
		\item[{\rm (I$_{0}$)}]  $ 0'' \approx 0$.
	\end{itemize}
$\mathcal{I}$ denotes the variety of implication zroupoids.  The varieties $\mathcal{DM}$ 
and $\mathcal{SL}$ 
are defined relative to $\mathcal{I}$, respectively, by the following identities: 
$$
\begin{array}{ll}
{\rm (DM)} & (x \to y) \to x \approx x \mbox{ (De Morgan Algebras)}; \\
{\rm (SL)}  & x' \approx x \mbox{ and } x \to y \approx y \to x \mbox{ (semilattices with the least element $0$)}.  
\end{array}
$$
The variety 
$\mathcal{BA}$ of Boolean algebras 
is defined relative to $\mathcal{DM}$ by the following identity: 
$$
\begin{array}{ll}
{\rm (BA)} & x \to x \approx 0'.\\
\end{array}
$$ 


The variety $\mathcal{I}$ 
 exhibits (see \cite{sankappanavarMorgan2012}) several interesting properties; 
for example, 
the identity $x''' \to y \approx x' \to y$ holds in $\mathcal{I}$; in particular, $x'''' \approx x''$.  

The (still largely unexplored) lattice of subvarieties of $\mathcal{I}$ seems to be fairly complex.  In fact, Problem 6 of \cite{sankappanavarMorgan2012} calls for an investigation of the structure of the lattice of subvarieties of $\mathcal{I}$.

The papers  \cite{cornejo2015implication}, \cite{cornejo2016order}, \cite{cornejo2016semisimple}, \cite{cornejo2016derived}, \cite{cornejo2016BolMoufang}, \cite{cornejo2016weakAssociative} and \cite{cornejo2016varieties} have addressed further the above-mentioned problem, but still partially, by introducing several new subvarieties of $\mathcal{I}$ and investigating relationships among them.  
The (currently known) size of the poset of subvarieties of $\mathcal{I}$ is at least 30; but it is still unknown whether the lattice of subvarieties is finite or infinite.  

The poset of the known varieties that appears in \cite{cornejo2016derived} is given below 
for the reader's convenience (for the definitions of the varieties in the picture, please see \cite{cornejo2016derived}). \\


\hskip 2.5cm
\setlength{\unitlength}{0.5cm}
\begin{picture}(16,30)(0,0)


\put(3,5){\circle*{0.3}}
\put(1,4.7){$\mathcal{BA}$}
\put(1.2,7){\tiny \cite{sankappanavarMorgan2012}}
\put(3,9){\circle*{0.3}}
\put(1,8.7){$\mathcal{KA}$}
\put(1.2,11){\tiny \cite{sankappanavarMorgan2012}}
\put(3,13){\circle*{0.3}}
\put(1,12.7){$\mathcal{DM}$}
\put(0.3,15){\tiny \cite[4.3]{sankappanavarMorgan2012}}
\put(3,17){\circle*{0.3}}
\put(1,16.7){$\mathcal{SCP}$}
\put(0.3,19){\tiny \cite[8.5]{sankappanavarMorgan2012}}
\put(6,20.2){\tiny \cite[8.19]{sankappanavarMorgan2012}}
\put(3,21){\circle*{0.3}}
\put(1,20.7){$\mathcal{CP}$}
\put(0.3,23){\tiny \cite[6.2]{cornejo2015implication}}
\put(4.8,23.7){\tiny \cite[8.10]{sankappanavarMorgan2012}}
\put(3,25){\circle*{0.3}}
 \put(-1.6,24.7){$\mathcal{WCP} = \mathcal{MC}$}

\put(7,1){\circle*{0.3}}
\put(5.5,0.7){$\mathcal{T}$}
\put(7,5){\circle*{0.3}}
\put(5.5,4.7){$\mathcal{SL}$}
\put(7,6.7){\tiny \cite{cornejo2015implication}}
\put(10.9,7.9){\tiny \cite{cornejo2015implication}}

\put(11,5){\circle*{0.3}}
\put(11.7,4.7){$\mathcal{Z}$}
\put(13.5,6.7){\tiny \cite{cornejo2015implication}}
\put(11,9){\circle*{0.3}}
\put(9.3,8.7){$\mathcal{I}_{1,0}$}
\put(8.2,11.7){\tiny \cite[10.3]{cornejo2015implication}}
\put(11,14){\circle*{0.3}}
\put(12.2,14.7){\tiny \cite{cornejo2015implication}} 
\put(11.7,13.7){$\mathcal{ID}$}
\put(11,15.7){\tiny \cite[8.23]{sankappanavarMorgan2012}}
\put(11,17){\circle*{0.3}}
\put(11.2,16.7){$\mathcal{LAP}$}
\put(11,17.7){\tiny \cite[8.19]{sankappanavarMorgan2012}}
\put(11,19){\circle*{0.3}}
\put(11.2,18.7){$\mathcal{I}_{2,0}$}
\put(11,21){\circle*{0.3}}
\put(11.2,20.7){$\mathcal{TII}$}
\put(8.3,22.7){\tiny \cite[8.3]{sankappanavarMorgan2012}}
\put(11,25){\circle*{0.3}}
\put(11.7,24.7){$\mathcal{I}_{3,1}$}
\put(11,29){\circle*{0.3}}
\put(11.7,28.7){$\mathcal{I}$}

\put(15,9){\circle*{0.3}}
\put(15.7,8.7){$\mathcal{C}$}
\put(15.2,10.7){\tiny \ref{subvarieties_011016}}
\put(4.5,19.7){\tiny \ref{subvarieties_011016}}
\put(15,13){\circle*{0.3}}
\put(15.7,12.7){$\mathcal{CLD}$}
\put(15.2,14.7){\tiny (*)}
\put(15,17){\circle*{0.3}}
\put(15.7,16.7){$\mathcal{A}$}
\put(12.3,23){\tiny \cite[6.2]{cornejo2015implication}}
\put(15.2,18.7){\tiny (*)}
\put(15,21){\circle*{0.3}}
\put(15.7,20.7){$\mathcal{SRD}$}
\put(15.2,22.7){\tiny \ref{Theo_SRD_RD}}
\put(15,25){\circle*{0.3}}
\put(15.7,24.7){$\mathcal{RD}$}


\put(3,5){\line(0,1){20}}
\put(7,1){\line(0,1){4}}
\put(11,9){\line(0,1){20}}
\put(15,9){\line(0,1){16}}


\put(3,5){\line(1,-1){4}}
\put(3,17){\line(1,-3){4}}
\put(3,17){\line(2,-3){8}}
\put(11,29){\line(1,-1){4}}
\put(11,25){\line(1,-2){4}}
\put(3,21){\line(1,-1){12}}


\put(3,25){\line(2,1){8}}
\put(3,17){\line(2,1){8}}
\put(3,21){\line(2,1){8}}
\put(3,13){\line(2,1){8}}
\put(7,5){\line(2,1){8}}
\put(7,5){\line(1,1){4}}
\put(11,9){\line(1,2){4}} 
\put(7,1){\line(1,1){8}}		

\end{picture}\\

Two of the subvarieties of $\mathcal{I}$ that are  of interest in this paper are:
 $\mathcal{I}_{2,0}$ and $\mathcal{MC}$ which are defined relative to $\mathcal{I}$, respectively, by the following identities, where $x \land y := (x \to y')' $ :  \\ 
\indent (I$_{2,0}$) \quad  $x'' \approx x$;\\
\indent (MC) \quad  $x \land y \approx y \land x$. 



Motivated by the fact that not all algebras in $\mathcal I$ are associative with respcet to the operation $\to$, the quest for finding more new subvarieties of $\mathcal I$ 
led us naturally to 
consider the question as to whether 
 generalizations 
 of the associative law would yield some new subvarieties of $\mathcal I$ and thereby reveal further insight into the structure of the lattice of subvarieties of $\mathcal I$.  This quest led to the results in  
 \cite{cornejo2016BolMoufang},  \cite{cornejo2016weakAssociative} and this paper, which will show that this is indeed the case.

A look at the associative law would reveal the following characteristics: 
\begin{itemize}
\item[(1)] Length of the left side term = length of the right side term = 3,
\item[(2)] The number of distinct variables on the left = the number of distinct variables on right = the number of occurrences of variables on either side,
\item[(3)] The order of the variables on the left side is the same as the order of the variables on the right side, 
\item[(4)] The bracketings used in the left side term and in the right side term are two different 
brackettings. 
\end{itemize}

One way to generalize the associative law is to relax somewhat the restrictions (1) and (2) by choosing $m$ distinct variables and setting the length of the left term = that of right term = $n$, with $n \geq m$,
and keeping (3) and (4).  But then, for $n \geq 4$, there will be more than two possible bracketings.  So, we order all possible bracketings and assign a number to each, called bracketing number.  Such identities are called ``weak associative identities of length $n$''. 
For a precise definition and notation of weak associative identities, we refer the reader to   \cite{cornejo2016weakAssociative} and the references therein.

A second way to generalize the associative law is to relax (3) and to keep (2),  (4) and the first half of (1). 
So, we consider the laws of the form $p \approx q$ of length $n$ such that (a) each of $p$ and $q$ contains the $n$ (an integer $\geq 3$) distinct variables, say, $x_1, x_2, \dots, x_n$, (b) $p$ and $q$ are terms obtained by distinct bracketings of  permutations of the $n$ variables.  Let us call such laws as ``identities of associative type of length $n$''.  


%

A third way to generalize the associative law is to relax all of four features mentioned above by allowing number of occurrences of variables on one side be different from the number on the other side.  Let us refer to these as ``identities of mixed type''.

Specific instances of all such generalizations of the associative law, 
have occurred in the literature at least since late 19th century.  We mention below a few such instances. 

 Weak associative identities 
of length $4$ with $3$ distinct variables, called ``identities of Bol-Moufang type'', have been investigated in the literature quite extensively  
for the varieties of quasigroups and loops.  In fact, the first systematic analysis of the relationships among the identities of Bol-Moufang type appears to be in \cite{Fe69} in the context of loops.  For more information about these identities in the context of quasigroups and loops, see \cite{Fe69}, \cite{Ku96}, \cite{PV05a}, \cite{PV05b}) and the references therein. 


 More recently, in \cite{cornejo2016BolMoufang} and \cite{cornejo2016weakAssociative}, we made  a complete analysis of relationships among weak associative identities of length $\leq 4$, relative to the variety  
$\mathcal S$ of symmetric I-zroupoids (i.e., satisfying $x'' \approx x$ and $(x \to y')' \approx (y \to x')'$). 
 We have shown 
that 
 $6$
 of the 155 subvarieties of $\mathcal S$ defined by a single weak associative law of length $m \leq 4$ (including the Bol-Moufang type) are distinct.
Furthermore, we describe explicitly by a Hasse diagram the poset formed by them,  together with the varieties $\mathcal{BA}$ and $\mathcal{SL}$. 

 We should mention here that such an analysis of weak associative laws of lenth $\leq 4$ relative to the variety $\mathcal I$ is still open.

 The identities of associative type have also appeared in the literature.
For example, 
\begin{itemize}
\item The identity $x \cdot (y \cdot z) \approx (z \cdot x) \cdot y$ 
was considered in  \cite{Su1929}  by Suschkewitsch (see also \cite[Theorem 11.5]{St1957}). 
\item Abbot uses the identity 
 $x \cdot (y \cdot z)\approx y \cdot (x \cdot z)$ as one of the defining identities in his definition of implication algebras. 
\item  
The identities $x \cdot(y \cdot z) \approx z \cdot (y \cdot x)$, $x \cdot(y \cdot z) \approx y \cdot (x \cdot z)$, and $x \cdot(y \cdot z) \approx (z \cdot x) \cdot y$  were investigated for quasigroups by Hossuz\'{u} in \cite{Ho54}. 
\item The identity $x \cdot(z \cdot y) \approx (x \cdot y) \cdot z$ is investigated by Pushkashu in
\cite{Pu10}. 
\item The identities $x \cdot(z \cdot y) \approx (x \cdot y) \cdot z$ and $x \cdot(y \cdot z) \approx z \cdot (y \cdot x)$ have appeared in \cite{KaNa72} of  Kazim and Naseeruddin. 
\end{itemize}





The identities of mixed type have also been considered in the literature.  A few are listed below:
\begin{itemize}
\item  left distributivity: $x \cdot (y \cdot z) \approx  (x \cdot y) \cdot (x \cdot z)$, 
 appears already in the late 19th century publications of logicians Peirce and Schroeder (see \cite{Pe1880} and \cite{Sch1887},  respectively). 
\item right distributive:  $(z \cdot y) \cdot x \approx (z \cdot x) \cdot (y \cdot x)$,
\item distributive if it is both left and right distributive,
\item  medial: $(x \cdot y) \cdot (u \cdot v) \approx (x \cdot u) \cdot (y \cdot v)$,
\item  idempotent:  $x \cdot  x \approx x$,
\item  left involutory (or left symmetric): $x \cdot (x \cdot y) \approx y$,
\end{itemize}




Several identities of associative type have appeared in the literatiure on groupoids as well.  For instance,
\begin{itemize}
\item  $(x \cdot y) \cdot z \approx (z \cdot y) \cdot x$: Abel-Grassmann's groupoid (AG-groupoid),
\item  $(x \cdot y) \cdot z \approx (z \cdot y) \cdot x$ and $x \cdot (y \cdot z) \approx y \cdot  (x \cdot z)$  (AG$^{**}$-groupoid),  
\item  $x \cdot (z \cdot y) \approx (x \cdot y) \cdot z$: Hossuzu-Tarski identity  (see \cite{Pu10}),
\item   $(x \cdot y) \cdot z \approx (z \cdot y) \cdot x$: Left almost semigroup (LA-semigroup) (see \cite{Pu10}),
\item   $x \cdot (y \cdot z) \approx z \cdot (y \cdot x)$: Right almost semigroup (RA-semigroup) (see \cite{Pu10}).
\end{itemize}

Similar to the problem mentioned in \cite{cornejo2016weakAssociative} for weak associative identities, the following general problem presents itself naturally if we restrict our attention to identities of associative type.\\

\noindent PROBLEM:  Let $\mathcal V$ be a given variety of algebras (whose language includes a   binary operation symbol, say, $`\to'$).  
Investigate the mutual relationships among the subvarieties of $\mathcal V$, each of which is  defined by a single identity of  
 associative type of length n, for small values of the positive integer $n$. \\

 We will now consider the above problem for the variety $\mathcal I$. We begin a systematic analysis of the relationships among the identities of associative type of length 3 relative to the variety $\mathcal I$.

\begin{definition}
An identity $p \approx q$, in the groupoid language $\langle \to \rangle$, is called {\it an identity of associative type of length $3$} if 
 $p$ and $q$ have exactly $3$ (distinct) variables, say x,y,z, and these variables are grouped according to one of the following 
two ways of grouping:
\begin{multicols}{2}
\begin{enumerate}[{\rm (a)}]
	\item $o \to (o \to o)$
	\item $(o \to o) \to o.$
\end{enumerate}
\end{multicols}
\end{definition}

In the rest of the paper, we refer to an ``identity of associative type of length $3$'' as simply 
an ``identity of associative type''.

We wish to determine the mutual relationships of all the subvarieties of $\mathcal I$ defined by 
the identities of associative type, which will be referred to as ``subvarieties of associative type''.  

 Our main theorem says that there are 8 of such subvarieties  of $\mathcal{I}$ that are
	distinct from each other and describes explicitly, by a Hasse diagram, the poset formed by them, together with the varieties $\mathcal{SL}$ and $\mathcal{BA}$. As an application,  we show that there are $3$ distinct subvarieties of $\mathcal S$ of associative type.


We would like to acknowledge that the software ``Prover 9/Mace 4'' developed by McCune \cite{Mc} has been useful to us in some of our findings presented in this paper.  We have used it to find examples and to check some conjectures.\\

\section{Preliminaries}

\indent We refer the reader to the standard references
\cite {balbesDistributive1974}, \cite{burrisCourse1981} and \cite{rasiowa1974algebraic} for concepts and results used, but not explained, in this paper.

Recall from \cite{sankappanavarMorgan2012} that $\mathcal{SL}$ is the variety of semilattices with a least element $0$. It was shown in \cite{cornejo2015implication} that $\mathcal{SL}   
= \mathcal{C} \cap \mathcal{I}_{1,0} $,  where $\mathcal{I}_{1,0}$ is defined by $x' \approx x$ relative to $\mathcal{I}$.

The two-element algebras $\mathbf{2_z}$, $\mathbf{2_s}$, $\mathbf{2_b}$ were introduced in \cite{sankappanavarMorgan2012}.  Their  
operations $\to$ are respectively as follows:  \\

\begin{minipage}{0.3 \textwidth}
	\begin{tabular}{r|rr}
		$\to$: & 0 & 1\\
		\hline
		0 & 0 & 0 \\
		1 & 0 & 0
	\end{tabular}  
	
\end{minipage}  
\begin{minipage}{0.3 \textwidth}
	\begin{tabular}{r|rr}
		$\to$: & 0 & 1\\
		\hline
		0 & 0 & 1 \\
		1 & 1 & 1
	\end{tabular} 
	
\end{minipage}
\begin{minipage}{0.3 \textwidth}
	\begin{tabular}{r|rr}
		$\to$: & 0 & 1\\
		\hline
		0 & 1 & 1 \\
		1 & 0 & 1
	\end{tabular}   
\end{minipage}           
\\ \\

Recall that $\mathcal{V}(\mathbf{2_b}) = \mathcal{BA}$.
Recall also from \cite[Corollary 10.4]{cornejo2015implication} \label{CorSL} that
$\mathcal{V}(\mathbf{2_s}) = \mathcal{SL}$.  The following lemma easily follows from the definition of $\land$ given earlier in the introduction.

\begin{Lemma} {\bf \cite[7.16]{sankappanavarMorgan2012}} \label{Lem_200317_02}
	Let $\mathbf A$ be an  $\mathcal{I}$-zroupoid. Then $\mathbf A \models x''' \to y \approx x' \to y$.
\end{Lemma}

\begin{Lemma} \label{lemaOfIZI} {\bf \cite[3.4]{cornejo2015implication}}  
	Let $\mathbf A$ be an I-zroupoid.  Then $\mathbf A$ satisfies:
	\begin{enumerate}[{\rm (a)}]
		\item $(x \to y) \to z \approx [(x \to y) \to z]'' $, \label{lemaOfIZI3_4}
		\item $(x \to y)' \approx (x'' \to y)' $. \label{lemaOfIZI2_4}
	\end{enumerate}
\end{Lemma}

\begin{Lemma} {\bf \cite[8.15]{sankappanavarMorgan2012}} \label{general_properties_equiv}
	Let $\mathbf A$ be an  $\mathcal{I}$-zroupoid. Then the following are equivalent:
	\begin{enumerate}
		\item $0' \to x \approx x$, \label{TXX} 
		\item $x'' \approx x$,
		\item $(x \to x')' \approx x$, \label{reflexivity}
		\item $x' \to x \approx x$. \label{LeftImplicationwithtilde}
	\end{enumerate}
\end{Lemma}

Recall that $\mathcal{I}_{2,0}$ and $\mathcal{MC}$ are the subvarieties of $\mathcal I$, defined, respectively,  by the equations
\begin{equation} \label{eq_I20} \tag{$I_{2,0}$}
x'' \approx x.
\end{equation}
\begin{equation} \label{eq_MC} \tag{MC}
x \wedge y \approx y \wedge x.
\end{equation}

\begin{Lemma}{\bf \cite{sankappanavarMorgan2012}} \label{general_properties}
	Let $\mathbf A \in  \mathcal{I}_{2,0}$. Then  
	\begin{enumerate}
		\item $x' \to 0' \approx 0 \to x$, \label{cuasiConmutativeOfImplic2}
		\item $0 \to x' \approx x \to 0'$. \label{cuasiConmutativeOfImplic}
	\end{enumerate}
\end{Lemma}

\begin{Lemma} \label{general_properties3} \label{general_properties2}
	Let $\mathbf A \in \mathcal{I}_{2,0}$. Then $\mathbf A$ satisfies:
		\begin{enumerate}[{\rm (a)}]
			\item $(x \to 0') \to y \approx (x \to y') \to y$, \label{281014_05}
			\item $(y \to x) \to y \approx (0 \to x) \to y$, \label{291014_10}
			\item $0 \to x \approx 0 \to (0 \to x)$, \label{311014_03}
			\item $(0 \to x) \to (0 \to y) \approx x \to (0 \to y)$, \label{311014_06}
			\item $x \to y \approx x \to (x \to y)$, \label{031114_04} 
			\item $0 \to (x \to y) \approx x \to (0 \to y)$, \label{071114_04}
			\item $0 \to (x \to y')' \approx 0 \to (x' \to y)$, \label{191114_05}
			\item $x \to (y \to x') \approx y \to x'$, \label{281114_01}
		\end{enumerate}
\end{Lemma}

\begin{proof}

	For items (\ref{281014_05}),
	 (\ref{291014_10}), (\ref{311014_03}), 
	 (\ref{071114_04}), 
	 (\ref{191114_05}), and   
	 (\ref{281114_01})
	 we refer the reader to \cite{cornejo2015implication}.
	The proof of items 
	(\ref{311014_06}) and (\ref{031114_04})
	are in \cite{cornejo2016order}.
\end{proof}

\begin{Theorem}{\bf \cite{cornejo2016derived}}  \label{Transfer_Theo}
	
	Let $t_i(\overline x), i= 1, \cdots, 6$
	be terms and $\mathcal V$ a subvariety of $\mathcal I$.   
	If 
	$$\mathcal V  \cap \mathcal I_{2,0} \models [t_1(\overline x) \to t_2(\overline x)] \to t_3(\overline x) \approx [t_4(\overline x) \to t_5(\overline x)] \to t_6(\overline x),$$ 
	then 
	$$\mathcal V\models [t_1(\overline x) \to t_2(\overline x)] \to t_3(\overline x) \approx [t_4(\overline x) \to t_5(\overline x)] \to t_6(\overline x).$$
\end{Theorem}

\subsection{Identities of Associative Type}

We now turn our attention to identities of associative type of length 3.  Recall that such an identity will contain three distinct variables that occur in any order and that are grouped in one of the two (obvious) ways.  The following identities play a crucial role in the sequel.

Let $\Sigma$ denote the set consisting of the following $14$ identities of associative type (of length 3 in the binary language $\langle \to \rangle$):

\begin{multicols}{2}
\begin{enumerate}[(${A}1$)]
\item   $x \to (y \to z) \approx (x \to y) \to z$ \label{identidadAsoc}\\   (Associative law)
\item  $x \to (y \to z) \approx x \to (z \to y)$
\item   $x \to (y \to z) \approx (x \to z) \to y$ 
\item   $x \to (y \to z) \approx y \to (x \to z)$
\item  $x \to (y \to z) \approx (y \to x) \to z$
\item   $x \to (y \to z) \approx y \to (z \to x)$
\item   $x \to (y \to z) \approx (y \to z) \to x$
\item  $x \to (y \to z) \approx  (z \to x) \to y$
\item   $x \to (y \to z) \approx z \to (y \to x)$
\item  $x \to (y \to z) \approx (z \to y) \to x$
\item   $(x \to y) \to z \approx (x \to z) \to y$ 
\item  $(x \to y) \to z \approx (y \to x) \to z$ 
\item   $(x \to y) \to z \approx (y \to z) \to x$ 
\item  $(x \to y) \to z \approx (z \to y) \to x$. 
\end{enumerate}
\end{multicols}

 We will denote by $\mathcal A_i$ the subvariety of $\mathcal I$ defined by the identity $(Ai)$, for $1 \leq i \leq 14$.  Such varieties will be referred to as subvarieties of $\mathcal I$ of associative type.
 
The following proposition is crucial for the rest of the paper.

\begin{Proposition} \label{propo_neces_identities} 
Let $\mathcal G$ be the variety of all groupoids of type $\{\to \}$  and
Let $\mathcal V$ denote the subvariety of $\mathcal G$ defined by a single identity of associative type.  Then $ V =\mathcal A_i$ for $i \in \{1, 2, \cdots, 14\}$.  
\end{Proposition} 

\begin{proof}


In an identity $p \approx q$ of associative type of length 3,  
 $p$ and $q$ have exactly 3 (distinct) variables, say x,y,z, and these variables are grouped according to one of the  
two ways of bracketing mentioned above.
Thus, there are six permuatations of 3 variables which give rise to the following 12 terms:

\begin{itemize}
\item[1a]: $x \to (y \to z)$ \quad \quad 1b:  $(x \to y) \to z$\\
\item[2a]: $x \to (z \to y)$ \quad \quad 2b:  $(x \to z) \to y$\\
\item[3a]:  $y \to (x \to z)$ \quad \quad 3b: $(y \to x) \to z$\\
\item[4a]:  $y \to (z \to x)$ \quad \quad 4b:  $(y \to z) \to x$\\
\item[5a]:  $z \to (x \to y)$ \quad \quad 5b: $( z \to x) \to y$\\
\item[6a]: $ z \to (y \to x)$ \quad \quad 6b: $(z \to y) \to x$.\\
\end{itemize} 

It is clear that these 12 terms, in turn, will lead to  
 66 identities in view of the symmetric property of equality.  It is routine to verify that each of the 66 identities is equivalent to one of the 14 identities of $\Sigma$ in the variety of groupoids.  Then the proposition follows. 
 \end{proof}



Our goal, in this paper, is to determine the distinct subvarieties of $\mathcal I$ and to  
  describe the poset of subvarieties of $\mathcal I$ of associative type.  It suffices to concentrate  on the varieties  defined by identities (A1)-(A14), in view of the above proposition.\\
 



\section{ Properties of subvarieties of $\mathcal I$ of Associative type}

 In this section we present properties of several subvarieties of $\mathcal I$ which will play a crucial role in our analysis of the identities of associative type relative to $\mathcal I$.

\begin{Lemma} \label{Lem_160317_01}
	Let $\mathbf A \in \mathcal{I}$ such that 
	$\mathbf A \models x' \to y \approx x \to y'$, 
	then $\mathbf A \models (x \to y) \to y' \approx x \to y'$.
\end{Lemma}

\begin{proof}
	Let $a,b \in A$. Then
\noindent $	(a \to b) \to b'	$
$\overset{  \ref{general_properties2} (\ref{281014_05})\, \& \, \ref{Transfer_Theo} 
}{=}  (a \to 0') \to b' $
$\overset{ hyp 
}{=}  (a' \to 0) \to b' $
$\overset{  
}{=}  a'' \to b' $
$\overset{ hyp 
}{=}  a''' \to b $
$\overset{  \ref{Lem_200317_02} 
}{=}  a' \to b $
$\overset{ hyp 
}{=}  a \to b' $.

\end{proof}

\begin{Lemma} \label{Lem_160317_02}
	Let $\mathbf A \in \mathcal{I}_{2,0}$ such that 
	$\mathbf A \models (x \to y)' \approx x \to (0 \to y)$, 
	then $\mathbf A \models x \to y' \approx x \to (0 \to y)$.
\end{Lemma}

\begin{proof}
	Let $a,b \in A$. Then
\noindent $	a \to b' $
$\overset{  \ref{general_properties_equiv} (\ref{TXX}) 
}{=}  a \to (0' \to  b)' $
$\overset{ hyp 
}{=}  a \to [0' \to  (0 \to b)] $
$\overset{  \ref{general_properties_equiv} (\ref{TXX})
}{=}  a \to (0 \to b) $.
\end{proof}

\begin{Lemma} \label{Lem_160317_03}
	Let $\mathbf A \in \mathcal{I}_{2,0}$ such that 
	$\mathbf A \models (x \to y)' \approx x \to (0 \to y)$.  Then\\ 
	 $\mathbf A \models [x \to (y \to z)']' \approx x \to (y \to (0 \to z))'$.
\end{Lemma}

\begin{proof}
	Let $a,b,c \in A$. We have that
%
\noindent $	[a \to (b \to c)']' $
$\overset{ hyp 
}{=}  a \to (0 \to (b \to c)') $
$\overset{ hyp 
}{=}  a \to (0 \to (b \to (0 \to c))) $
$\overset{  \ref{general_properties2} (\ref{071114_04}) 
}{=}  a \to (b \to (0 \to (0 \to c)))  $
$\overset{ hyp 
}{=}  a \to [b \to (0 \to c)'] $.
\end{proof}

\begin{Lemma} \label{Lem_200317_03}
	Let $\mathbf A \in \mathcal{I}$ such that $\mathbf A $ satisfies:
\begin{enumerate}[(1)]
	\item $(x \to y)' \approx x \to (0 \to y),$ \label{200317_04}
	\item $x' \to y \approx x \to y'$. \label{200317_05}
\end{enumerate}
	Then, $\mathbf A \models 0 \to [x \to (y \to z)] \approx 0 \to [(x \to y) \to z]$.
\end{Lemma}

\begin{proof}
Let $a,b, c \in A$. Then,
\noindent $0 \to [(a \to b) \to c] $
$\overset{  (I) 
}{=}  0 \to [(c' \to a) \to (b \to c)']' $
$\overset{  (I) 
}{=}  0 \to \{[(b \to c)'' \to c'] \to [a \to (b \to c)']\}'' $
$\overset{  \ref{Transfer_Theo} 
}{=}  0 \to \{[(b \to c) \to c'] \to [a \to (b \to c)']\}'' $
$\overset{  \ref{Lem_160317_01} 
}{=}  0 \to \{[b \to  c'] \to [a \to (b \to c)']\}'' $
$\overset{  \ref{Lem_160317_02} \, \& \, \ref{Transfer_Theo} \, \& \, hyp
}{=}  0 \to \{[b \to  (0 \to c)] \to [a \to (b \to c)']\}'' $
$\overset{  \ref{Lem_160317_03} \, \& \, hyp 
}{=}  0 \to \{[b \to  (0 \to c)] \to [a \to [b \to  (0 \to c)]']\}'' $
$\overset{  \ref{general_properties2} (\ref{281114_01})\, \& \, \ref{Transfer_Theo} 
}{=}  0 \to \{a \to [b \to  (0 \to c)]'\}'' $
$\overset{  (\ref{200317_04}) 
}{=}  0 \to \{a \to [b \to  c]''\}'' $
$\overset{  \ref{Transfer_Theo} 
}{=}  0 \to \{a \to [b \to  c]\}'' $
$\overset{  (\ref{200317_05}) 
}{=}  0' \to \{a \to [b \to  c]\}' $
$\overset{  (\ref{200317_05}) 
}{=}  0'' \to \{a \to [b \to  c]\} $
$\overset{  
}{=}  0 \to \{a \to [b \to  c]\}  $.
\end{proof}

\begin{Lemma} \label{Lem_280317_01}
	Let $\mathbf A \in \mathcal I$
	 such that 
	$\mathbf A \models (x \to y)' \approx (y \to x)'$.  Then\\ 
	 $\mathbf A \models (x \to y) \to z \approx (y \to x) \to z$.
\end{Lemma}

\begin{proof} Let $a,b,c \in A$. Hence
\noindent $(a \to b) \to c$
$\overset{  \ref{Transfer_Theo} 
}{=}  (a \to b)'' \to c $
$\overset{ hyp 
}{=}  (b \to a)'' \to c $
$\overset{  \ref{Transfer_Theo} 
}{=}  (b \to a) \to c $
\end{proof}

\begin{definition}
	Let $\mathbf A \in \mathcal I$. We say that $\mathcal A$ is of type 1 if the following identities hold in $\mathbf{A}$:
\begin{enumerate} [{(${E}1$)}]
	\item $(x \to y)' \approx x \to (0 \to y)$, \label{typeI_equation1}
	\item $x' \to y \approx x \to y'$, \label{typeI_equation2}
	\item  $0 \to (x \to y) \approx 0 \to (y \to x)$, \label{typeI_equation3}
	\item  $x \to (y \to z) \approx (p(x) \to p(y)) \to p(z)$, where $p$ is some permutation of 
	$\{x,y,z\}$.  \label{typeI_equation4}
\end{enumerate}
\end{definition}

\begin{Theorem} \label{Theo_230317_15}
If $\mathbf A \in \mathcal I$ is of type 1 then $\mathbf A \models (Aj)$ for all $1 \leq j \leq 14$.
	
\end{Theorem}

\begin{proof}
	Let $\mathbf A \in \mathcal I$ be of type 1, and $a,b,c \in A$.  In view of equations (${E}\ref{typeI_equation1}$), (${E}\ref{typeI_equation2}$) and Lemma \ref{Lem_200317_03} we have that
\begin{equation}\label{210317_01}
\mathbf A \models 0 \to [x \to (y \to z)] \approx 0 \to [(x \to y) \to z].
\end{equation}
Then we can consider the following cases.
\begin{itemize}
	\item Assume that $j=1$.  Then 
\noindent $a \to (b \to c)	$
$\overset{  (E\ref{typeI_equation4}) 
}{=}  (p(a) \to p(b)) \to p(c) $
$\overset{  \ref{lemaOfIZI} 
}{=}  [(p(a) \to p(b)) \to p(c)]'' $
$\overset{  (E\ref{typeI_equation1}) 
}{=}  [(p(a) \to p(b)) \to (0 \to p(c))]' $
$\overset{  \ref{general_properties2} (\ref{071114_04})\, \& \, \ref{Transfer_Theo} 
}{=}  [0 \to [(p(a) \to p(b)) \to p(c)]]' $
$\overset{ (E\ref{typeI_equation3}) \, \& \, (\ref{210317_01}) 
}{=}  [0 \to [(a \to b) \to c]]' $
$\overset{  \ref{general_properties2} (\ref{071114_04})  \, \& \, \ref{Transfer_Theo} 
}{=}  [(a \to b) \to [0 \to c]]' $
$\overset{  (E\ref{typeI_equation1}) 
}{=}  [(a \to b) \to c]'' $
$\overset{  \ref{lemaOfIZI} 
}{=}  (a \to b) \to c. $

The cases $j=3, 5, 7, 8, 10$ are similar.	

	\item Assume that $j=2$. Then, in the same way as in the case of $j = 1$ we have that
\begin{equation}\label{210317_02}
\mathbf A \models x \to (y \to z) \approx [0 \to [(p(x) \to p(y)) \to p(z)]]'.
\end{equation}
Then, 
\noindent $a \to (b \to c) $
$\overset{  (\ref{210317_02}) 
}{=}  [0 \to [(p(a) \to p(b)) \to p(c)]]' $
$\overset{(E\ref{typeI_equation3}) \, \& \, (\ref{210317_01}) 
}{=}  [0 \to [(p(a) \to p(c)) \to p(b)]]' $
$\overset{  \ref{general_properties2} (\ref{071114_04})  \, \& \, \ref{Transfer_Theo} 
}{=}  [(p(a) \to p(c)) \to [0 \to p(b)]]' $
$\overset{  (E\ref{typeI_equation1}) 
}{=}  [(p(a) \to p(c)) \to p(b)]'' $
$\overset{  \ref{lemaOfIZI} (\ref{lemaOfIZI3_4}) 
}{=}  (p(a) \to p(c)) \to p(b) $
$\overset{  (E\ref{typeI_equation4}) 
}{=}  (a \to c) \to b $.

The cases $j=4, 6, 9$ are similar.

\item Assume that $j=11$.
\noindent $	(a \to b) \to c	$
$\overset{  \ref{lemaOfIZI} (\ref{lemaOfIZI3_4})
}{=}   [(a \to b) \to c]'' $
$\overset{  (E\ref{typeI_equation1}) 
}{=}  [(a \to b) \to (0 \to c)]' $
$\overset{  \ref{general_properties2} (\ref{071114_04})  \, \& \, \ref{Transfer_Theo} 
}{=}  [0 \to [(a \to b) \to c]]' $
$\overset{ (E\ref{typeI_equation3}) \, \& \, (\ref{210317_01}) 
}{=}  [0 \to [(a \to c) \to b]]' $
$\overset{  \ref{general_properties2} (\ref{071114_04})  \, \& \, \ref{Transfer_Theo} 
}{=}  [(a \to c) \to [0 \to b]]' $
$\overset{  (E\ref{typeI_equation1}) 
}{=}  [(a \to c) \to b]'' $
$\overset{  \ref{lemaOfIZI} 
}{=}  (a \to c) \to b. $

	The cases $j=12, 13, 14$ are similar.

\end{itemize}
\end{proof}

To prove that the variety $\mathcal A_j$ is of type 1, with $j \in \{3,5,7,8,10\}$, we need the following lemmas.

\begin{Lemma} \label{230317_11}
	If $\mathbf A \in \mathcal A_5 $ then $\mathbf A$ satisfies
	\begin{enumerate}[(a)]
		\item $x' \to y \approx x \to y'$ \label{210317_05}
		\item $(x \to y)' \approx 0 \to (x \to y)$ \label{210317_06}
		\item $x \to (0 \to y) \approx 0 \to (x \to y)$ \label{210317_07}
	\end{enumerate}
\end{Lemma}

\begin{proof}
	Let $a,b \in A$. Then
	\begin{enumerate}[(a)]
		\item 	
\noindent $		a \to b' $
$\overset{  
}{=}  a \to (b \to 0) $
$\overset{  (A5) 
}{=}  (b \to a) \to 0 $
$\overset{  \ref{Transfer_Theo} 
}{=}  (b \to a'') \to 0 $
$\overset{  
}{=}  (b \to (a' \to 0)) \to 0 $
$\overset{  (A5) 
}{=}  [(a' \to b) \to 0] \to 0 $
$\overset{  
}{=}  (a' \to b)'' $
$\overset{  \ref{Transfer_Theo}
}{=}  a' \to b $.

\item
Observe that
\noindent $(a \to b)' $
$\overset{  
}{=}  (a \to b) \to 0 $
$\overset{  (A5) 
}{=}  b \to (a \to 0) $
$\overset{  
}{=}  b \to a' $
$\overset{  \ref{230317_11} (\ref{210317_05}) 
}{=}  b' \to a $
$\overset{  
}{=}  (b \to 0) \to a $
$\overset{  (A5) 
}{=}  0 \to (b \to a) $.
\item 
Notice that 
\noindent $0 \to (a \to b) $
$\overset{  (A5) 
}{=}  (a \to 0) \to b $
$\overset{  \ref{Transfer_Theo} 
}{=}  (a'' \to 0) \to b $
$\overset{  \ref{230317_11} (\ref{210317_05}) 
}{=}  (a' \to 0') \to b $
$\overset{  \ref{general_properties} (\ref{cuasiConmutativeOfImplic}) \, \& \, \ref{Transfer_Theo} 
}{=}  (0 \to a) \to b $
$\overset{  (A5) 
}{=}  a \to (0 \to b) $.
	\end{enumerate}	
\end{proof}

\begin{Lemma} \label{230317_12}
	If $\mathbf A \in \mathcal A_8 $ then $\mathbf A$ satisfies:
	\begin{enumerate}[(a)]
		\item $x \to y' \approx x' \to y'$, \label{230317_01}
		\item $ x \to y' \approx 0 \to (y' \to x)$. \label{230317_02}
	\end{enumerate}
\end{Lemma}

\begin{proof}
Let $a,b \in A$. Then
\noindent $ a \to b' $
$\overset{  
}{=}  a \to (b \to 0) $
$\overset{  (A8) 
}{=}  (0 \to a) \to b $
$\overset{  (I) 
}{=}  [(b' \to 0) \to (a \to b)']' $
$\overset{  
}{=}  [(b' \to 0) \to (a \to b)'] \to 0 $
$\overset{  (A8) 
}{=}  (a \to b)' \to (0 \to (b' \to 0)) $
$\overset{  (A8) 
}{=}  (a \to b)' \to ((0 \to 0) \to b') $
$\overset{  
}{=}  (a \to b)' \to (0' \to b') $
$\overset{  \ref{general_properties_equiv} (\ref{TXX}) \, \& \, \ref{Transfer_Theo} 
}{=}  (a \to b)' \to b' $
$\overset{  
}{=}  ((a \to b) \to 0) \to b' $
$\overset{  (A8) 
}{=}  0 \to (b' \to (a \to b)) $
$\overset{  (A8) 
}{=}  0 \to ((b \to b') \to a) $
$\overset{  \ref{Transfer_Theo} 
}{=}  0 \to ((b'' \to b') \to a) $
$\overset{  \ref{general_properties_equiv} (\ref{LeftImplicationwithtilde}) \, \& \, \ref{Transfer_Theo} 
}{=}  0 \to (b' \to a) $, implying that $\mathbf A$ satisfies the identity (b).
Next,
$0 \to (b' \to a) $
$\overset{  (A8) 
}{=}  (a \to 0) \to b' $
$\overset{  
}{=}  a' \to b' $, thus $\mathbf A$ satisfies the identity (a).


\end{proof}

\begin{Lemma} \label{230317_13}
	If $\mathbf A \in \mathcal A_{10}$ then $\mathbf A$ satisfies:
	\begin{enumerate}[(a)]
		\item $[0 \to (x \to y)]' \approx x \to y'$, \label{230317_06}
		\item $(y \to x)'' \approx x \to y'$, \label{230317_07}
		\item $(x \to y)' \approx x \to y'$. \label{230317_08}
	\end{enumerate}
\end{Lemma}

\begin{proof} Let $a,b \in A$.
	\begin{enumerate}[(a)]
		\item
We have that 
$$
\mathbf A \models [0 \to (x \to y)]' \approx x \to y',
$$
since
$$
\begin{array}{lcll}
a \to b' & = & a \to (b \to 0) & \mbox{} \\
& = & (0 \to b) \to a & \mbox{by (A10)} \\
& = & [(a' \to 0) \to (b \to a)']' & \mbox{by (I)} \\
& = & [((a \to 0) \to 0) \to (b \to a)']' & \mbox{} \\
& = & [(0 \to (0 \to a)) \to (b \to a)']' & \mbox{by (A10)} \\
& = & [(0 \to a) \to (b \to a)']' & \mbox{by \ref{general_properties2} (\ref{311014_03})  and \ref{Transfer_Theo}} \\
& = & [(0 \to a) \to [(b \to a) \to 0]]' & \mbox{} \\
& = & [(0 \to a) \to [0 \to (a \to b)]]' & \mbox{by (A10)} \\
& = & [(0 \to a) \to [(0 \to a) \to (0 \to b)]]' & \mbox{by \ref{general_properties2} (\ref{071114_04}) and (\ref{311014_06}) and by \ref{Transfer_Theo}} \\
& = & [(0 \to a) \to (0 \to b)]' & \mbox{by \ref{general_properties2} (\ref{031114_04}) and \ref{Transfer_Theo}} \\
& = & [0 \to (a \to b)]' & \mbox{by \ref{general_properties2} (\ref{311014_06}),  (\ref{071114_04}) 
and by \ref{Transfer_Theo}}. 
\end{array}
$$
\item Observe that
\noindent $a \to b' $
$\overset{  \ref{230317_13} (\ref{230317_06}) 
}{=}   [0 \to (a \to b)]' $
$\overset{  (A10) 
}{=}  [(b \to a) \to 0]' $
$\overset{  
}{=}  (b \to a)'' $.
Hence,  $\mathbf A \models {\rm(b)}$.
\item Since
\noindent $a \to b' $
$\overset{  \ref{230317_13} (\ref{230317_06}) 
}{=}   [0 \to (a \to b)]' $
$\overset{  \ref{general_properties2} (\ref{311014_03}) \, \& \, \ref{Transfer_Theo} 
}{=}  [0 \to (0 \to (a \to b))]' $
$\overset{  (A10) 
}{=}  [((a \to b) \to 0) \to 0]' $
$\overset{  
}{=}  (a \to b)''' $
$\overset{  
}{=}  (a \to b)', $ we conclude that $\mathbf A \models {\rm(c)}$.
	\end{enumerate}
\end{proof}

\begin{Lemma} \label{Lem_230317_14}
	If $\mathbf A \in \mathcal A_3 \cup \mathcal A_5 \cup \mathcal A_7 \cup \mathcal A_8 \cup \mathcal A_{10}$ then $\mathbf A$ satisfies
\begin{enumerate}[{\rm (1)}]
	\item $(x \to y)' \approx x \to (0 \to y)$, \label{300717_01}
	\item $x' \to y \approx x \to y'$ and \label{300717_02}
	\item $0 \to (x \to y) \approx 0 \to (y \to x)$.  \label{300717_03}
\end{enumerate}
\end{Lemma}

\begin{proof}
	Let $a,b \in A$.
\begin{itemize}
	\item Suppose $\mathbf A \in \mathcal A_3$.  Then
\noindent $(a \to b)' $
$\overset{  
}{=}  (a \to b) \to 0 $
$\overset{  (A3) 
}{=}  a \to (0 \to b) $, implying $\mathbf A \models (\ref{300717_01})$.

Observe that
\noindent $a' \to b $
$\overset{  
}{=}  (a \to 0) \to b $
$\overset{  (A3) 
}{=}  a \to (b \to 0) $
$\overset{  
}{=}  a \to b'. $ So, (\ref{300717_02}) holds in $\mathbf A$.
Also, 
\noindent $0 \to (a \to b) $
$\overset{  (A3) 
}{=}  (0 \to b) \to a $
$\overset{  (I) 
}{=}  [(a' \to 0) \to (b \to a)']' $
$\overset{  
}{=}  [a'' \to (b \to a)']' $
$\overset{  \ref{Transfer_Theo} 
}{=}  [a \to (b \to a)']' $
$\overset{  (\ref{Lem_230317_14}) 
}{=}  [a' \to (b \to a)]' $
$\overset{  \ref{general_properties2} (\ref{281114_01}) \, \& \, \ref{Transfer_Theo} 
}{=}  (b \to a)' $
$\overset{  \ref{general_properties_equiv} (\ref{TXX}) \, \& \, \ref{Transfer_Theo} 
}{=}  0' \to (b \to a)' $
$\overset{ \ref{Lem_230317_14}(\ref{300717_02}) 
}{=}{ 0'' \to (b \to a) }$
$\overset{  
}{=}  0 \to (b \to a),$ proving (\ref{300717_03}).

\item Assume that $\mathbf A \in \mathcal A_5$. 
Then
\noindent $(a \to b)' $
$\overset{  \ref{230317_11} (\ref{210317_06}) 
}{=}   0 \to (b \to a) $
$\overset{  
}{=}  0'' \to (b \to a) $
$\overset{  \ref{230317_11} (\ref{210317_05}) 
}{=}  0' \to (b \to a)' $
$\overset{  \ref{general_properties_equiv} (\ref{TXX}) \, \& \, \ref{Transfer_Theo} 
}{=}  (b \to a)' $
$\overset{  
}{=}  (b \to a) \to 0 $
$\overset{  (A5) 
}{=}  a \to (b \to 0) $
$\overset{  
}{=}  a \to b' $
$\overset{  \ref{230317_11} (\ref{210317_05}) 
}{=}  a' \to b $
$\overset{  \ref{general_properties_equiv} (\ref{TXX}) \, \& \, \ref{Transfer_Theo} 
}{=}  (0' \to a') \to b $
$\overset{  \ref{230317_11} (\ref{210317_05}) 
}{=}  (0 \to a'') \to b $
$\overset{  \ref{Transfer_Theo} 
}{=}  (0 \to a) \to b $
$\overset{(A5)  
}{=} a \to (0 \to b)$, proving (\ref{300717_01}).
(\ref{300717_02}) is immediate from Lemma \ref{230317_11} (\ref{210317_05}).
Next,  
\noindent $0 \to (a \to b) $
$\overset{  \ref{230317_11} (\ref{210317_06}) 
}{=}  (a \to b)' $
$\overset{  
}{=}  (a \to b) \to 0 $
$\overset{  (A5) 
}{=}  b \to (a \to 0) $
$\overset{  
}{=}  b \to a' $
$\overset{  \ref{230317_11} (\ref{210317_05}) 
}{=}  b' \to a $
$\overset{  \ref{Transfer_Theo} 
}{=}  (b' \to a)'' $
$\overset{  \ref{230317_11} (\ref{210317_05}) 
}{=}  (b \to a')'' $
$\overset{ \ref{230317_11} (\ref{210317_06}) 
}{=}  [0 \to (b \to a')]' $
$\overset{  \ref{230317_11} (\ref{210317_07}) 
}{=}  [b \to (0 \to a')]' $
$\overset{  \ref{230317_11} (\ref{210317_05}) 
}{=}  [b \to (0' \to a)]' $
$\overset{  \ref{general_properties_equiv} (\ref{TXX}) \, \& \, \ref{Transfer_Theo} 
}{=}  (b \to a)' $
$\overset{  \ref{230317_11} (\ref{210317_06}) 
}{=}  0 \to (b \to a) $, proving (\ref{300717_03}).

\item Assume that $\mathbf A \in \mathcal A_7$.  Then
\noindent $(a \to b)' $
$\overset{  \ref{Transfer_Theo} 
}{=}  (a \to b)''' $
$\overset{  
}{=}  [(a \to b) \to 0]'' $
$\overset{  (A7) 
}{=}  [0 \to (a \to b)]'' $
$\overset{  \ref{general_properties2} (\ref{071114_04}) \, \& \, \ref{Transfer_Theo} 
}{=}  [a \to (0 \to b)]'' $
$\overset{  (A7) 
}{=}  [(0 \to b) \to a]'' $
$\overset{  \ref{Transfer_Theo} 
}{=}  (0 \to b) \to a $
$\overset{  (A7) 
}{=}  a \to (0 \to b),$ proving (\ref{300717_01}).

Next,  
\noindent $a' \to b $
$\overset{  \ref{Transfer_Theo} 
}{=}  [a' \to b]'' $
$\overset{  
}{=}  [(a' \to b) \to 0]' $
$\overset{  (A7) 
}{=}  [0 \to (a' \to b)]' $
$\overset{  \ref{Transfer_Theo} 
}{=}  [0 \to (a' \to b'')]'  $
$\overset{  \ref{general_properties2} (\ref{191114_05}) \, \& \, \ref{Transfer_Theo} 
}{=}  [0 \to (a \to b')']' $
$\overset{  
}{=}  [0 \to ((a \to b') \to 0)]' $
$\overset{  (A7) 
}{=}  [0 \to (0 \to (a \to b'))]' $
$\overset{  \ref{general_properties2} (\ref{311014_03}) \, \& \, \ref{Transfer_Theo} 
}{=}  [0 \to (a \to b')]' $
$\overset{  
}{=}  [0 \to (a \to b')] \to 0 $
$\overset{  (A7) 
}{=}  [(a \to b') \to 0] \to 0  $
$\overset{  
}{=}  (a \to b')'' $
$\overset{  
}{=}  (a \to (b \to 0))'' $
$\overset{  (A7) 
}{=}  ((b \to 0) \to a)'' $
$\overset{  \ref{Transfer_Theo} 
}{=}  (b \to 0) \to a $
$\overset{  (A7) 
}{=}  a \to (b \to 0) $
$\overset{  
}{=}  a \to b',$ proving (\ref{300717_02}).
Finally, observe that 
\noindent $0 \to (a \to b) $
$\overset{  (A7) 
}{=}  (a \to b) \to 0 $
$\overset{  \ref{Transfer_Theo} 
}{=}  (a'' \to b) \to 0 $
$\overset{  
}{=}  [(a' \to 0) \to b] \to 0 $
$\overset{  (A7) 
}{=}  [b \to (a' \to 0)] \to 0 $
$\overset{  
}{=}  [b \to a''] \to 0 $
$\overset{  \ref{Transfer_Theo} 
}{=}  [b \to a] \to 0 $
$\overset{  (A7) 
}{=}  0 \to (b \to a),$ proving (\ref{300717_03}).

\item Let $\mathbf A \in \mathcal A_8$.
We will prove (\ref{300717_02}) first. Now,
\noindent $a \to b' $
$\overset{  \ref{230317_12} (\ref{230317_02}) 
}{=}  0 \to (b' \to a) $
$\overset{  \ref{Transfer_Theo} 
}{=}  0 \to (b' \to a'') $
$\overset{  \ref{230317_12} (\ref{230317_01}) 
}{=}  0 \to (b'' \to a'') $
$\overset{  (A8) 
}{=}  (a'' \to 0) \to b'' $
$\overset{  \ref{Transfer_Theo} 
}{=}  (a'' \to 0) \to b $
$\overset{  
}{=}  a''' \to b $
$\overset{  \ref{Lem_200317_02} 
}{=}  a' \to b,$ proving (\ref{300717_02}).

Notice that
\noindent $a \to (0 \to b) $
$\overset{  (A8) 
}{=}  (b \to a) \to 0 $
$\overset{  
}{=}  (b \to a)' $
$\overset{  \ref{Transfer_Theo} 
}{=}  (b \to a)''' $
$\overset{  
}{=}  (b \to a)'' \to 0 $
$\overset{ (\ref{300717_02})                       
}{=} (b \to a)' \to 0' $
$\overset{  \ref{general_properties_equiv} (\ref{TXX}) \, \& \, \ref{Transfer_Theo} 
}{=}  [0' \to (b \to a)'] \to 0' $
$\overset{ (\ref{300717_02})        
}{=}   [0'' \to (b \to a)] \to 0' $
$\overset{  
}{=}  [0 \to (b \to a)] \to 0' $
$\overset{  \ref{general_properties2} (\ref{071114_04}) \, \& \, \ref{Transfer_Theo} 
}{=}  [b \to (0 \to a)] \to 0' $
$\overset{  (A8) 
}{=}  [(a \to b) \to 0] \to 0' $
$\overset{  
}{=}  (a \to b)' \to 0' $
$\overset{ (\ref{300717_02})                       
}{=}  (a \to b) \to 0'' $
$\overset{  
}{=}  (a \to b) \to 0 $
$\overset{  
}{=}  (a \to b)',$ proving (\ref{300717_01}).

The identity 
\begin{equation} \label{230317_05}
0 \to (x \to y) \approx (x \to y)'
\end{equation}
holds in $\mathbf A$ since
\noindent $0 \to (a \to b) $
$\overset{  
}{=}  0'' \to (a \to b) $
$ \overset{ (\ref{300717_02})            
}{=}  0' \to (a \to b)' $
$\overset{  \ref{general_properties_equiv} (\ref{TXX}) \, \& \, \ref{Transfer_Theo} 
}{=}  (a \to b)' $.

Then 
\noindent $0 \to (a \to b) $
$\overset{  (\ref{230317_05}) 
}{=}  (a \to b)' $
$\overset{  
}{=}  (a \to b) \to 0 $
$\overset{  (A8) 
}{=}  b \to (0 \to a) $
$ \overset{ (\ref{300717_01})   
}{=}  (b \to a)' $
$\overset{  (\ref{230317_05}) 
}{=}  0 \to (b \to a),$ proving (\ref{300717_03}).

\item Assume that $\mathbf A \in \mathcal A_{10}$.
Hence 
$$
\begin{array}{lcll}
a' \to b & = & a' \to b'' & \mbox{by \ref{Transfer_Theo}} \\
& = & (a' \to b')' & \mbox{by \ref{230317_13} (\ref{230317_08})} \\
& = & (a' \to b') \to 0 & \mbox{} \\
& = & 0 \to (b' \to a') & \mbox{by (A10)} \\
& = & 0 \to (b' \to (a \to 0)) & \mbox{} \\
& = & 0 \to ((0 \to a) \to b') & \mbox{by (A10)} \\
& = & (b' \to (0 \to a)) \to 0 & \mbox{by (A10)} \\
& = & (b' \to (0 \to a))' & \mbox{} \\
& = & ((b \to 0) \to (0 \to a))' & \mbox{} \\
& = & [(0 \to a) \to (0 \to b)]' & \mbox{by (A10)} \\
& = & [0 \to (a \to b)]' & \mbox{by \ref{general_properties2} (\ref{071114_04}) \, \& \, (\ref{311014_06}) and by \ref{Transfer_Theo}} \\
& = & [(b \to a) \to 0]' & \mbox{by (A10)} \\
& = & (b \to a)'' & \mbox{} \\
& = & a \to b' & \mbox{by \ref{230317_13} (\ref{230317_07})}, 
\end{array}
$$
proving (\ref{300717_02}).

Consider
\noindent $a \to (0 \to b) $
$\overset{  (A10) 
}{=}  (b \to 0) \to a $
$\overset{  (I) 
}{=}  [(a' \to b) \to (0 \to a)']' $
$\overset{  \ref{230317_13} (\ref{230317_08}) 
}{=}  [(a' \to b) \to (0 \to a')]' $
$\overset{ (\ref{300717_02})     
}{=}  [(a' \to b) \to (0' \to a)]' $
$\overset{  \ref{general_properties_equiv} (\ref{TXX}) \, \& \, \ref{Transfer_Theo} 
}{=}  [(a' \to b) \to a]' $
$\overset{ (\ref{300717_02})                            
}{=}  [(a \to b') \to a]' $
$\overset{  \ref{general_properties2} (\ref{291014_10}) \, \& \, \ref{Transfer_Theo} 
}{=}  [(0 \to b') \to a]' $
$\overset{  (A10) 
}{=}  [a \to (b' \to 0)]' $
$\overset{  
}{=}  [a \to b'']' $
$\overset{  \ref{Transfer_Theo} 
}{=}  (a \to b)' $, proving (\ref{300717_01}).

To finish off the proof,
\noindent $0 \to (a \to b) $
$\overset{  (A10) 
}{=}  (b \to a) \to 0 $
$\overset{  \ref{230317_13} (\ref{230317_08}) 
}{=}  b \to a' $
$\overset{ (\ref{300717_02})      
}{=}  b' \to a $
$\overset{  
}{=}  (b \to 0) \to a $
$\overset{  (A10) 
}{=}  a \to (0 \to b) $
$\overset{ (\ref{300717_01})                
}{=} (a \to b)' $
$\overset{  
}{=}  (a \to b) \to 0 $
$ \overset{  (A10) 
}{=}  0 \to (b \to a) $.

\end{itemize}
\end{proof}

\begin{Theorem} \label{Theo_280317_05}
$\mathcal A_3 = \mathcal A_5 = \mathcal A_7 = \mathcal A_8 = \mathcal A_{10}$.
\end{Theorem}

\begin{proof}
Let $\mathbf A \in \mathcal A_3 \cup \mathcal A_5 \cup \mathcal A_7 \cup \mathcal A_8 \cup \mathcal A_{10}$. By Lemma \ref{Lem_230317_14} we have that $\mathbf A$ is of type 1. Then, using Theorem \ref{Theo_230317_15}, $\mathbf A \in \mathcal A_j$ for all $j \in \{3,5,7,8,10\}$.
\end{proof}

\begin{Lemma} \label{270317_01}
	If $\mathbf A \in \mathcal A_{13}$ then $\mathbf A$ satisfies
	\begin{enumerate}[(a)]
		\item $(x \to y)' \approx (0 \to x) \to y$ \label{270317_02}
		\item $(x \to y)' \approx x' \to y'$ \label{270317_03}
		\item $(x \to y)' \approx (0 \to y) \to x'$ \label{270317_04}
		\item $(x \to y)' \approx (x \to y)''$ \label{270317_05}
		\item $(x \to y)' \approx (y \to x)'$ \label{270317_06}
	\end{enumerate}
\end{Lemma}

\begin{proof}
Let us consider $a,b \in A$. 
	\begin{enumerate}[(a)]
		\item 
\noindent $(a \to b)' $
$\overset{  
}{=}  (a \to b) \to 0 $
$\overset{  (A13) 
}{=}  (b \to 0) \to a $
$\overset{  (A13) 
}{=}  (0 \to a) \to b $.
Hence $\mathbf{A} \models {\rm(a)}$.

\item 
Observe that
\noindent $(a \to b)' $
$\overset{  \ref{Transfer_Theo} 
}{=}  0' \to (a \to b)' $
$\overset{  
}{=}  (0  \to 0) \to (a \to b)' $
$\overset{  (A13) 
}{=}  [0 \to (a \to b)'] \to 0 $
$ \overset{  (I20) \, \& \, \ref{Transfer_Theo}
}{=}   [0 \to (a \to b'')'] \to 0 $
$ \overset{  \ref{general_properties2} (\ref{191114_05}) \, \& \, \ref{Transfer_Theo} 
}{=}   [0 \to (a' \to b')] \to 0 $
$\overset{  (A13) 
}{=}  [(a' \to b') \to 0] \to 0 $
$\overset{  
}{=}  (a' \to b')'' $
$\overset{  \ref{Transfer_Theo} 
}{=} a' \to b'  $



\item 
Observe
\noindent $(a \to b)' $
$\overset{  (\ref{270317_02}) 
}{=}  (0 \to a) \to b $
$\overset{  \ref{Transfer_Theo} 
}{=}  (0 \to a)'' \to b  $
$\overset{  (\ref{270317_03}) 
}{=}  (0' \to a')' \to b $
$\overset{  \ref{Transfer_Theo} 
}{=}  a'' \to b $
$\overset{  (A13) 
}{=}  (0 \to b) \to a'. $

\item 
Note that
\noindent $(a \to b)' $
$\overset{  (\ref{270317_04}) 
}{=}  (0 \to b) \to a' $
$\overset{  (I) 
}{=}  [(a'' \to 0) \to (b \to a')']' $
$\overset{  
}{=}  [a''' \to (b \to a')']' $
$\overset{  \ref{Lem_200317_02} 
}{=}  [a' \to (b \to a')']' $
$\overset{  (\ref{270317_03}) 
}{=}  [a' \to (b' \to a'')]' $
$\overset{  \ref{Transfer_Theo} 
}{=}  [a' \to (b' \to a)]' $
$\overset{  \ref{general_properties2} (\ref{281114_01}) \, \& \, \ref{Transfer_Theo} 
}{=}  (b' \to a)' $
$\overset{  (A13) 
}{=}  ((0 \to a) \to b)' $
$\overset{  (A13) 
}{=}  ((a \to b) \to 0)' $
$\overset{  
}{=}  (a \to b)''. $

\item We have
\noindent $(b \to a)' $
$\overset{  (\ref{270317_04}) 
}{=}  (0 \to a) \to b' $
$\overset{  \ref{Transfer_Theo} 
}{=}  [(0 \to a) \to b']'' $
$\overset{  (\ref{270317_03}) 
}{=}  [(0 \to a)' \to b'']' $
$\overset{  (\ref{270317_05}) 
}{=}  [(0 \to a)'' \to b'']' $
$\overset{  \ref{Transfer_Theo} 
}{=}  [(0 \to a) \to b]' $
$\overset{  (\ref{270317_05}) 
}{=}  [(0 \to a) \to b]'' $
$\overset{  \ref{Transfer_Theo} 
}{=}  (0 \to a) \to b $
$\overset{  (\ref{270317_02}) 
}{=}  (a \to b)' .$
	\end{enumerate}
\end{proof}

\begin{Theorem} \label{Theo_280317_03}
	$\mathcal A_{11} = \mathcal A_{12} = \mathcal A_{13}$.
\end{Theorem} 

\begin{proof}Let us consider $\mathbf A \in \mathcal A_{11}$ and $a,b,c \in A$. Hence
\noindent $(a \to b) \to c $
$\overset{  \ref{general_properties_equiv} (\ref{TXX}) \, \& \, \ref{Transfer_Theo} 
}{=}  ((0' \to a) \to b) \to c $
$\overset{  (A11) 
}{=}  ((0' \to b) \to a) \to c $
$\overset{  \ref{general_properties_equiv} (\ref{TXX}) \, \& \, \ref{Transfer_Theo}. 
}{=}  (b \to a) \to c $.
Hence,  $\mathbf A \in \mathcal A_{12}$, implying $A_{11} \subseteq A_{12}.$

Now assume that $\mathbf A \in \mathcal A_{12}$ and $a,b,c \in A$. Then
\noindent $(a \to b) \to c $
$\overset{  (I) 
}{=}  [(c' \to a) \to (b \to c)']' $
$\overset{  (I) 
}{=}  \{[(b \to c)'' \to c'] \to [a \to (b \to c)']'\}'' $
$\overset{   \ref{Transfer_Theo} 
}{=}  \{[(b \to c) \to c'] \to [a \to (b \to c)']'\}'' $
$\overset{  (A12) 
}{=}  \{[c' \to (b \to c)] \to [a \to (b \to c)']'\}'' $
$\overset{  \ref{general_properties2} (\ref{281114_01}) \, \& \, \ref{Transfer_Theo} 
}{=}  \{(b \to c) \to [a \to (b \to c)']'\}'' $
$\overset{  (A12) 
}{=}  \{[a \to (b \to c)']' \to (b \to c)\}'' $
$\overset{  
}{=}  \{[[a \to (b \to c)'] \to 0] \to (b \to c)\}'' $
$\overset{  (A12) 
}{=}  \{[0 \to [a \to (b \to c)']] \to (b \to c)\}'' $
$\overset{  \ref{general_properties2} (\ref{291014_10}) \, \& \, \ref{Transfer_Theo} 
}{=}  \{[(b \to c) \to [a \to (b \to c)']] \to (b \to c)\}'' $
$\overset{  \ref{general_properties2} (\ref{281114_01}) \, \& \, \ref{Transfer_Theo} 
}{=}  \{[a \to (b \to c)'] \to (b \to c)\}'' $
$\overset{  \ref{general_properties2} (\ref{281014_05}) \, \& \, \ref{Transfer_Theo} 
}{=}  \{[a \to 0'] \to (b \to c)\}'' $
$\overset{  (A12) 
}{=}  \{[0' \to a] \to (b \to c)\}'' $
$\overset{  \ref{general_properties_equiv} (\ref{TXX}) \, \& \, \ref{Transfer_Theo} 
}{=}  \{a \to (b \to c)\}'' $
$\overset{  (A12) 
}{=}  \{(b \to c) \to a\}'' $
$\overset{   \ref{Transfer_Theo} 
}{=}  (b \to c) \to a $, which implies that
 $A_{12} \subseteq A_{13}.$
 
If $\mathbf A \in \mathcal A_{13}$ and $a,b,c \in A$, then
\noindent $(a \to b) \to c $
$\overset{  (A13) 
}{=}  (b \to c) \to a $
$\overset{  (A13) 
}{=}  (c \to a) \to b $
$\overset{  \ref{Transfer_Theo} 
}{=}  (c \to a)'' \to b $
$\overset{  \ref{270317_01} (\ref{270317_06}) 
}{=}  (a \to c)'' \to b $
$\overset{  \ref{Transfer_Theo} 
}{=}  (a \to c) \to b $,
concluding that
$A_{13} \subseteq A_{11}.$
\end{proof}

\section{Main Theorem} \label{main_theorem_section}



In this section we will prove our main theorem.  But first we need one more lemma.

\begin{Lemma} \label{Lem_280317_02}
	If $\mathbf A \in \mathcal A_2 \cup \mathcal A_6 \cup \mathcal A_9$ then $\mathbf A \in \mathcal A_{11}$.
\end{Lemma}

\begin{proof}
We will see that $\mathbf A \models (x \to y)' \approx (y \to x)'$. 

Let $a,b \in A$.
\begin{itemize}
	\item If $\mathbf A \in \mathcal A_2$,
\noindent $(a \to b) \to 0 $
$\overset{  \ref{general_properties_equiv} (\ref{TXX}) \, \& \, \ref{Transfer_Theo} 
}{=}  (0' \to (a \to b)) \to 0 $
$\overset{  (A2) 
}{=}  (0' \to (b \to a)) \to 0 $
$\overset{  \ref{general_properties_equiv} (\ref{TXX}) \, \& \, \ref{Transfer_Theo}
}{=}  (b \to a) \to 0 $.

\item If $\mathbf A \in \mathcal A_6$,
\noindent $(a \to b) \to 0 $
$\overset{  \ref{general_properties_equiv} (\ref{TXX}) \, \& \, \ref{Transfer_Theo} 
}{=}   (0' \to (a \to b)) \to 0 $
$\overset{  (A6) 
}{=}  (a \to (b \to 0')) \to 0 $
$\overset{  (A6) 
}{=}  (b \to (0' \to a)) \to 0  $
$\overset{  \ref{general_properties_equiv} (\ref{TXX}) \, \& \, \ref{Transfer_Theo} 
}{=}  (b \to a) \to 0 $.

\item  If $\mathbf A \in \mathcal A_9$, then


$$
\begin{array}{lcll}
(a \to b)' & = & (a \to b) \to 0 & \mbox{} \\
& = & (a \to (0' \to b)) \to 0 & \mbox{by \ref{general_properties_equiv} (\ref{TXX})  and  \ref{Transfer_Theo}} \\
& = & (b \to (0' \to a)) \to 0 & \mbox{by (A9)} \\
& = &  (b \to a) \to 0 & \mbox{by \ref{general_properties_equiv} (\ref{TXX})  and  \ref{Transfer_Theo}} \\
& = &  (b \to a)' & \mbox{} 
\end{array}
$$
\end{itemize}

Now, apply Lemma \ref{Lem_280317_01}, to get $\mathbf A \in A_{12}$. Therefore, using Theorem \ref{Theo_280317_03},  we conclude $\mathbf A \in \mathcal A_{11}$. 

\end{proof}

We are now ready to present the main theorem of this paper.
\begin{Theorem} \label{mainTheorem} We have
\begin{enumerate}[{\rm(a)}]
	\item The following are the $8$ subvarieties of $\mathcal I$ of associative type that are distinct from each other. \label{280317_07}
	$$\mathcal A_1, \mathcal A_2, \mathcal A_3, \mathcal A_4,  \mathcal A_6, \mathcal A_9, \mathcal A_{11} \mbox{ and } \mathcal A_{14}.$$
	\item They satisfy the following relationships: \label{280317_06}
	\begin{enumerate}[1.]
		\item $\mathcal{SL} \subset \mathcal A_3 \subset \mathcal A_4$, 
		\item $\mathcal{BA} \subset \mathcal A_4 \subset \mathcal I$, 
		\item $\mathcal A_3 \subset \mathcal A_1 \subset \mathcal I$,
		\item $\mathcal A_3 \subset \mathcal A_2 \subset \mathcal A_{11}$, 
		$\mathcal A_3 \subset \mathcal A_6 \subset \mathcal A_{11}$ and
		$\mathcal A_3 \subset \mathcal A_9 \subset \mathcal A_{11}$,
		\item $ \mathcal A_{11} \subset \mathcal A_{14} \subset \mathcal I$.
	\end{enumerate}
\end{enumerate}
\end{Theorem}

\begin{proof}
Observe that, in view of Theorem \ref{Theo_280317_05} and Theorem \ref{Theo_280317_03} we can conclude that each of the $14$ subvarieties of associative type of $\mathcal I$ is equal to one of the following varieties:
	$$\mathcal A_1, \mathcal A_2, \mathcal A_3, \mathcal A_4,  \mathcal A_6, \mathcal A_9, \mathcal A_{11}, \mathcal A_{14}.$$ 
	We first wish to prove (\ref{280317_06}).  Notice that by Lemma \ref{Lem_230317_14} we have that $\mathbf A \in \mathcal A_3$ is of type 1. Then, using Theorem \ref{Theo_230317_15}, $\mathbf A \in \mathcal A_j$ for all $1 \leq j \leq 14$. Hence 
\begin{equation} \label{280317_08}
\mathcal A_3 \subseteq \mathcal A_j \mbox{ for all } 1 \leq j \leq 14.
\end{equation}	

	\begin{enumerate}[1.]
	\item Recall that $\mathcal{SL} = \mathcal C \cap I_{1,0}$. Then, we get $\mathcal C \subseteq \mathcal A_1$ and $\mathcal I_{1,0} \subseteq \mathcal A_1$  by \cite[Theorem 8.2]{cornejo2015implication} and 
	\cite[Theorem 9.3]{cornejo2015implication}, respectively, implying $\mathcal{SL} \subseteq \mathcal A_3$, and $\mathcal A_3 \subseteq \mathcal A_4$ by (\ref{280317_08}).
		 
The algebras $\mathbf{2_z}$ and $\mathbf{2_b}$ show that 
$\mathcal{SL} \not= \mathcal A_3$ and $\mathcal A_3 \not= \mathcal A_4$, respectively.
		 


	\item In view of \cite{cornejo2016weakAssociative} we have that $\mathcal{BA} \subset \mathcal S$. By \cite[Lemma 3.1]{cornejo2016BolMoufang}, 
		$\mathcal S \models x \to (y \to z) \approx y \to (x \to z)$. Thus, $\mathcal{BA} \ \subseteq\ \mathcal A_4$.

The algebra $\mathbf{2_s}$ shows that $\mathcal{BA} \not= \mathcal A_4$ 
and 
	the following algebra shows that 
	$\mathcal A_4 \not= \mathcal I$, respectively. 
	

	\begin{tabular}{r|rrr}
		$\to$: & 0 & 1 & 2\\
		\hline
		0 & 0 & 0 & 0 \\
		1 & 2 & 0 & 0 \\
		2 & 0 & 0 & 0
	\end{tabular}

\item The algebra $\mathbf{2_b}$ shows that $\mathcal A_1 \not= \mathcal I$ and the following algebra witnesses that $\mathcal A_3 \not= \mathcal A_1$.   

\begin{tabular}{r|rrr}
	$\to$: & 0 & 1 & 2\\
	\hline
	0 & 0 & 1 & 2 \\
	1 & 1 & 1 & 2 \\
	2 & 2 & 1 & 2
\end{tabular}


\item Using \eqref{280317_08}  and Lemma \ref{Lem_280317_02} we can conclude that $\mathcal A_3 \subseteq \mathcal A_2 \subseteq \mathcal A_{11}$, 
$\mathcal A_3 \subseteq \mathcal A_6 \subseteq \mathcal A_{11}$ and
$\mathcal A_3 \subseteq \mathcal A_9 \subseteq \mathcal A_{11}$.

The following algebras show that $\mathcal A_3 \not= \mathcal A_2$ and $\mathcal A_2 \not= \mathcal A_{11}$, respectively.

\begin{tabular}{r|rrr}
	$\to$: & 0 & 1 & 2\\
	\hline
	0 & 0 & 0 & 0 \\
	1 & 2 & 0 & 2 \\
	2 & 0 & 0 & 0
\end{tabular}

\begin{tabular}{r|rrr}
	$\to$: & 0 & 1 & 2\\
	\hline
	0 & 0 & 0 & 0 \\
	1 & 2 & 0 & 0 \\
	2 & 0 & 0 & 0
\end{tabular}
	
The following algebras show that $\mathcal A_3 \not= \mathcal A_6$ and $\mathcal A_6 \not= \mathcal A_{11}$, respectively.

\begin{tabular}{r|rrrr}
	$\to$: & 0 & 1 & 2 & 3\\
	\hline
	0 & 0 & 0 & 0 & 0 \\
	1 & 0 & 2 & 3 & 0 \\
	2 & 0 & 0 & 0 & 0 \\
	3 & 0 & 0 & 0 & 0
\end{tabular}

\begin{tabular}{r|rrr}
	$\to$: & 0 & 1 & 2\\
	\hline
	0 & 0 & 0 & 0 \\
	1 & 2 & 0 & 0 \\
	2 & 0 & 0 & 0
\end{tabular} 
		
The following algebras show that $\mathcal A_3 \not= \mathcal A_9$ and $\mathcal A_9 \not= \mathcal A_{11}$, respectively.

\begin{tabular}{r|rrrr}
	$\to$: & 0 & 1 & 2 & 3\\
	\hline
	0 & 0 & 0 & 0 & 0 \\
	1 & 0 & 2 & 3 & 0 \\
	2 & 0 & 0 & 0 & 0 \\
	3 & 0 & 0 & 0 & 0
\end{tabular}

\begin{tabular}{r|rrr}
	$\to$: & 0 & 1 & 2\\
	\hline
	0 & 0 & 0 & 0 \\
	1 & 2 & 0 & 0 \\
	2 & 0 & 0 & 0
\end{tabular}

\item Let $\mathbf A \in \mathcal A_{11}$ and $a,b,c \in A$. By Theorem \ref{Theo_280317_03}, $\mathcal A_{11} = \mathcal A_{12} = \mathcal A_{13}$. Hence
\noindent $(a \to b) \to c $
$\overset{  (A13) 
}{=}  (b  \to c) \to a $
$\overset{  (A12) 
}{=}  (c \to b) \to a. $
Therefore $\mathcal A_{11} \subseteq \mathcal A_{14}$.

The algebra $\mathbf{2_b}$ shows that $\mathcal A_{14} \not= \mathcal I$ and
the following algebra shows that $\mathcal A_{11} \not= \mathcal A_{14}$.  

\begin{tabular}{r|rrrr}
	$\to$: & 0 & 1 & 2 & 3\\
	\hline
	0 & 0 & 1 & 2 & 3 \\
	1 & 2 & 3 & 2 & 3 \\
	2 & 1 & 1 & 3 & 3 \\
	3 & 3 & 3 & 3 & 3
\end{tabular}


	\end{enumerate}	
The proof of the theorem is now complete since (\ref{280317_07}) is an immediate consequence of 
 (\ref{280317_06}).
\end{proof}


The Hasse diagram of the poset of subvarieties of $\mathcal I$ of associative type, together with $\mathcal{SL}$ and $\mathcal{BA}$, is:

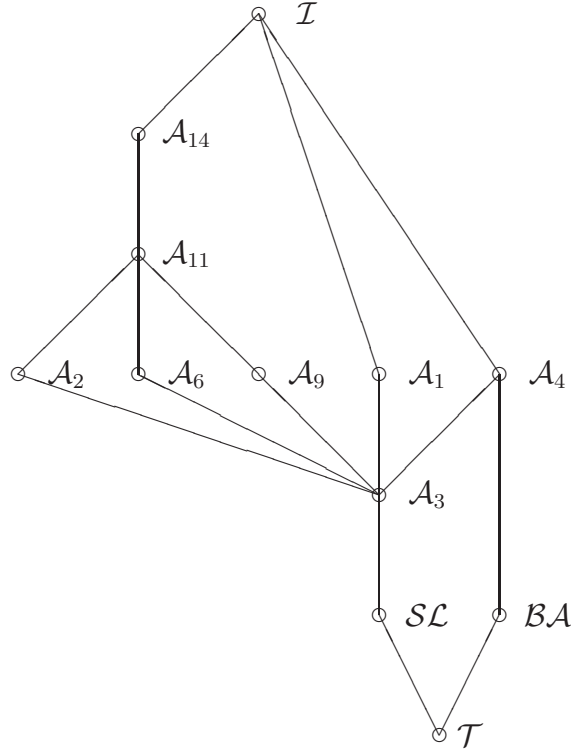
\begin{minipage}{0.5 \textwidth}
	\setlength{\unitlength}{8mm}
	\begin{picture}(12,15)(0,0)
\put(9,2){\circle{0.2}}

\put(8,4){\circle{0.2}}
\put(10,4){\circle{0.2}}

\put(8,6){\circle{0.2}}

\put(2,8){\circle{0.2}}	
\put(4,8){\circle{0.2}}	
\put(6,8){\circle{0.2}}	
\put(8,8){\circle{0.2}}	
\put(10,8){\circle{0.2}}

\put(4,10){\circle{0.2}}	

\put(4,12){\circle{0.2}}	

\put(6,14){\circle{0.2}}

\put(4,12){\line(1,1){2}}
\put(2,8){\line(1,1){2}}
\put(8,6){\line(1,1){2}}
\put(9,2){\line(1,2){1}}

\put(4,8){\line(0,1){4}}
\put(8,4){\line(0,1){4}}
\put(10,4){\line(0,1){4}}

\put(8,8){\line(-1,3){2}}
\put(10,8){\line(-2,3){4}}
\put(6,8){\line(-1,1){2}}
\put(9,2){\line(-1,2){1}}
\put(8,6){\line(-3,1){6}}
\put(8,6){\line(-2,1){4}}
\put(8,6){\line(-1,1){2}}

\put(9.5,2){\makebox(0,0){$\mathcal T$}}

\put(8.8,4){\makebox(0,0){$\mathcal{SL}$}}
\put(10.8,4){\makebox(0,0){$\mathcal{BA}$}}

\put(8.8,6){\makebox(0,0){$\mathcal{A}_3$}}

\put(2.8,8){\makebox(0,0){$\mathcal{A}_2$}}
\put(4.8,8){\makebox(0,0){$\mathcal{A}_6$}}
\put(6.8,8){\makebox(0,0){$\mathcal{A}_9$}}
\put(8.8,8){\makebox(0,0){$\mathcal{A}_1$}}
\put(10.8,8){\makebox(0,0){$\mathcal{A}_4$}}

\put(4.8,10){\makebox(0,0){$\mathcal{A}_{11}$}}

\put(4.8,12){\makebox(0,0){$\mathcal{A}_{14}$}}

\put(6.8,14){\makebox(0,0){$\mathcal I$}}

	\end{picture}
\end{minipage}


\section{Identities of Associative type in Symmetric Implication Zroupoids}

	Let $\mathbf{A} \in \mathcal{I}$.  $\mathbf{A}$ is {\rm involutive} if $\mathbf{A} \in \mathcal{I}_{2,0}$.  $\mathbf{A}$ is {\rm meet-commutative} if $\mathbf{A} \in \mathcal{MC}$.  $\mathbf{A}$ is {\rm symmetric} if $\mathbf{A}$ is both involutive and meet-commutative.    Let $\mathcal{S}$ denote the variety of symmetric $\mathcal{I}$-zroupoids.  In other words, $\mathcal{S} = \mathcal{I}_{2,0} \cap \mathcal{MC}$.
The variety $\mathcal{S}$ was investigated in \cite{cornejo2015implication}, \cite{cornejo2016BolMoufang} and \cite{cornejo2016weakAssociative} and has some interesting properties.

In this section we give an application of the main theorem, Theorem 4.2, to describe the poset of the subvarieties of the variety $\mathcal{S}$.




\begin{Lemma} {\rm \cite[Lemma 3.1 (a)]{cornejo2016BolMoufang} \label{properties_of_I20_MC}}
	Let $\mathbf A  \in \mathcal S$. Then $\mathbf A$ satisfies
$x \to (y \to z) \approx y \to (x \to z)$.
\end{Lemma}

\begin{Lemma}{\rm \cite[Lemma 2.1]{cornejo2016BolMoufang} \label{lemma_SL_I10_C}}
	$\mathcal{MC} \cap \mathcal{I}_{1,0} \subseteq  \mathcal{C} \cap \mathcal{I}_{1,0} = \mathcal{SL}$.  
\end{Lemma}

\begin{Lemma}{\rm \cite[Lemma 3.2]{cornejo2016BolMoufang}} \label{310516_09}
	Let $\mathbf A  \in \mathcal S$ such that $\mathbf A \models x \to x \approx x$. Then $\mathbf A \models x' \approx x$. 
\end{Lemma}

\begin{Lemma} \label{110717_03}
	$\mathcal A_{11} \cap \mathcal S= \mathcal{SL}$.
\end{Lemma}

\begin{proof}
	Let $\mathbf A \in \mathcal A_{11} \cap \mathcal S$ and $a \in A$. Since $\mathcal S \subseteq \mathcal I_{2,0}$, we have
$$
\begin{array}{lcll}
a & = & a' \to a & \mbox{by Lemma \ref{general_properties_equiv} (\ref{LeftImplicationwithtilde})} \\
& = & (0' \to a') \to a & \mbox{by Lemma \ref{general_properties_equiv} (\ref{TXX})} \\
& = & (0' \to a) \to a' & \mbox{by (A11)} \\
& = & a \to a' & \mbox{by Lemma \ref{general_properties_equiv} (\ref{TXX})} \\
& = & a'' \to a' & \mbox{} \\
& = & a' & \mbox{by Lemma \ref{general_properties_equiv} (\ref{LeftImplicationwithtilde}).} 
\end{array}
$$
Therefore, $\mathbf A \models x \approx x'$. Then, by Lemma \ref{lemma_SL_I10_C}, $\mathbf A \in \mathcal{SL}$.
\end{proof}

\begin{Lemma} \label{110717_04}
	$\mathcal A_{1} \cap \mathcal S \subseteq \mathcal{SL}$.
\end{Lemma}

\begin{proof}
	Let $\mathbf A \in \mathcal A_{1} \cap \mathcal S$ and $a \in A$. 
Then
$$
\begin{array}{lcll}
a & = & 0'\to a & \mbox{by Lemma \ref{general_properties_equiv} (\ref{TXX})} \\
& = & (0 \to 0) \to a & \mbox{} \\
& = & 0 \to (0 \to a) & \mbox{by (A1)} \\
& = & 0 \to a & \mbox{by Lemma \ref{general_properties2} (\ref{311014_03})}. 
\end{array}
$$
Consequently,
\begin{equation} \label{110717_01}
\mathbf A \models x \approx 0 \to x.
\end{equation}	
Therefore,
$$
\begin{array}{lcll}
a & = & a' \to a & \mbox{by Lemma \ref{general_properties_equiv} (\ref{LeftImplicationwithtilde})} \\
& = & (a \to 0) \to a & \mbox{} \\
& = & a \to (0 \to a) & \mbox{by (A1)} \\
& = & a \to a & \mbox{by equation (\ref{110717_01})} 
\end{array}
$$
Thus, by Lemma \ref{310516_09}, $\mathbf A \models x' \approx x$. Using Lemma \ref{lemma_SL_I10_C} we can conclude the proof.

\end{proof}

We will denote by $\mathcal S_i$ the variety $\mathcal A_i \cap \mathcal S$ with $1 \leq i \leq 14$.

We will now give an application of the main theorem, Theorem \ref{mainTheorem}, of Section \ref{main_theorem_section}.  

\begin{Proposition} \label{110717_07}
	Each of the $14$ subvarieties of associative type of $\mathcal S$ is equal to one of the following varieties:
	$$\mathcal{SL}, \,   \mathcal S_{14}, \, \mathcal S.$$
\end{Proposition}

\begin{proof}
	From Theorem \ref{Theo_280317_05} and Theorem \ref{Theo_280317_03}  we know that each of the $14$ subvarieties of associative type of $\mathcal I$ is equal to one of the following varieties:
	$$\mathcal A_1, \mathcal A_2, \mathcal A_3, \mathcal A_4,  \mathcal A_6, \mathcal A_9, \mathcal A_{11}, \mathcal A_{14}.$$
	Using Theorem \ref{mainTheorem}, Lemma \ref{110717_03} and Lemma \ref{110717_04} we have that 
	$$\mathcal{SL} \subseteq \mathcal S_3 \subseteq \mathcal S_2 \subseteq \mathcal S_{11} \subseteq \mathcal{SL},$$
	$$\mathcal{SL} \subseteq \mathcal S_6 \subseteq  \mathcal S_{11} \subseteq \mathcal{SL},$$
	$$\mathcal{SL} \subseteq \mathcal S_9 \subseteq  \mathcal S_{11} \subseteq \mathcal{SL}$$ and
	$$\mathcal{SL} \subseteq \mathcal S_1 \subseteq  \mathcal{SL}$$

By Lemma \ref{properties_of_I20_MC}, $\mathcal A_4 = \mathcal S$, So, $\mathcal S_4 = \mathcal S.$
\end{proof}

We are now ready to present the main theorem of this section.
\begin{Theorem} \label{mainsectionTheorem} We have
	\begin{enumerate}[{\rm(a)}]
		\item The following are the $3$ subvarieties of $\mathcal S$ of associative type  that are distinct from each other. \label{110717_05}
		$$  \mathcal{SL}, \,  \mathcal S_{14}, \, \mathcal S.$$
		\item They satisfy the following relationships \label{110717_06}
		\begin{enumerate}[1.]
			\item $\mathcal{SL} \subset \mathcal S_{14} \subset \mathcal S,$
			\item $\mathcal{BA} \not\subset \mathcal S_{14}.$
		\end{enumerate}
	\end{enumerate}
\end{Theorem}
	
	\begin{proof}
		We first prove (\ref{110717_06}). 
		
		\begin{enumerate}[1.]
			\item By Theorem \ref{mainTheorem},  $\mathcal{SL} \subseteq \mathcal S_{14}$.
			
			The following algebras show that $\mathcal{SL} \not= \mathcal S_{14}$ and $\mathcal S_{14} \not= \mathcal S$, respectively.
			
\begin{tabular}{r|rrrr}
	$\to$: & 0 & 1 & 2 & 3\\
	\hline
	0 & 0 & 1 & 2 & 3 \\
	1 & 2 & 3 & 2 & 3 \\
	2 & 1 & 1 & 3 & 3 \\
	3 & 3 & 3 & 3 & 3
\end{tabular}

\begin{tabular}{r|rr}
	$\to$: & 0 & 1\\
	\hline
	0 & 1 & 1 \\
	1 & 0 & 1
\end{tabular}
			
			\item 
Since $\mathbf{2_b} \not \models (S_{14}$), it follows that $\mathcal{BA} \not \subseteq \mathcal {S}_{14}$.			
			
\end{enumerate}				

		The proof of the theorem is now complete since (\ref{110717_05}) is an immediate consequence of Proposition \ref{110717_07} and (\ref{110717_06}).

	\end{proof}
	
The Hasse diagram of the poset of subvarieties of $\mathcal S$ of associative type, together with  $\mathcal{BA}$, is:

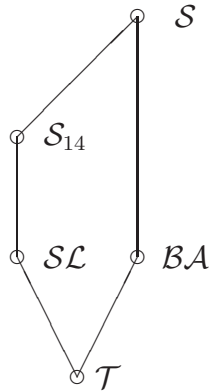
\begin{minipage}{0.5 \textwidth}
	\setlength{\unitlength}{8mm}
	\begin{picture}(5,10)(0,0)

	\put(3,2){\circle{0.2}}
	
	\put(2,4){\circle{0.2}}
	\put(4,4){\circle{0.2}}
	
	\put(2,6){\circle{0.2}}
	
	\put(4,8){\circle{0.2}}
	
	
	
	
	\put(2,6){\line(1,1){2}}
	\put(3,2){\line(1,2){1}}
	
	\put(2,4){\line(0,1){2}}
	\put(4,4){\line(0,1){4}}

	\put(3,2){\line(-1,2){1}}

	\put(3.5,2){\makebox(0,0){$\mathcal T$}}
	
	\put(2.8,4){\makebox(0,0){$\mathcal{SL}$}}
	\put(4.8,4){\makebox(0,0){$\mathcal{BA}$}}
	
	\put(2.8,6){\makebox(0,0){$\mathcal{S}_{14}$}}
	
	\put(4.8,8){\makebox(0,0){$\mathcal S$}}
	
	
	
	
	\end{picture}
\end{minipage}\\

\vspace{1cm}

\noindent{\bf Acknowledgment:} 
The first author wants to thank the
institutional support of CONICET  (Consejo Nacional de Investigaciones Cient\'ificas y T\'ecnicas).   
\ \\

\noindent {\bf Compliance with Ethical Standards:}\\ 

\noindent {\bf Conflict of Interest:} The first author declares that he has no conflict of interest. The second author  declares that he has no conflict of interest. \\

\noindent {\bf Ethical approval:}
This article does not contain any studies with human participants or animals performed by any of the authors. \\

	\noindent {\bf Funding:}  
	The work of Juan M. Cornejo was supported by CONICET (Consejo Nacional de Investigaciones Cientificas y Tecnicas) and Universidad Nacional del Sur. 
	Hanamantagouda P. Sankappanavar did not receive any specific grant from funding agencies in the public, commercial, or not-for-profit sectors.

\ \\ \ \\

\noindent {\sc Juan M. Cornejo}\\
Departamento de Matem\'atica\\
Universidad Nacional del Sur\\
Alem 1253, Bah\'ia Blanca, Argentina\\
INMABB - CONICET\\

\noindent jmcornejo@uns.edu.ar

\vskip 1.4cm

\noindent {\sc Hanamantagouda P. Sankappanavar}\\
Department of Mathematics\\
State University of New York\\
New Paltz, New York 12561\\
U.S.A.\\

\noindent sankapph@newpaltz.edu


\small

\begin{thebibliography}{99}

   \bibitem[Ab] {Ab67}  Abbott JC (1967). Semi-boolean algebra. Matematicki Vesnik 4(19), 40 p. 177-198. 

  \bibitem[BD74]{balbesDistributive1974}
    Balbes R,  Dwinger, PH (1974).  
    Distributive lattices.
    Univ. of Missouri Press, Columbia.
     
    
 \bibitem[Be34]{bernstein1934postulates}
 
 Bernstein BA (1934). A set of four postulates for Boolean algebras in terms of the implicative operation.
Trans. Amer. Math. Soc. 36 , 876-884.
	
	\bibitem[BS81]{burrisCourse1981} 
    Burris S,  Sankappanavar HP (1981). A Course in universal algebra. 
    Springer-Verlag,  New York.
    The free, corrected version (2012) is available online as a PDF file 
    at  {\sf math.uwaterloo.ca/~snburris}.  
	
\bibitem[CS16a]{cornejo2015implication}
Cornejo JM, Sankappanavar HP (2016). On implicator groupoids. Algebra univers. 
	77(2), 125--146.  

\bibitem[CS16b]{cornejo2016order}
Cornejo JM, Sankappanavar HP (2016). Order in implication zroupoids. Stud Logica, 104(3), 417-453.

\bibitem[CS16c]{cornejo2016semisimple}
Cornejo JM, Sankappanavar HP (2016). Semisimple varieties of implication zroupoids.  Soft Comput.  20, p. 3139 -- 3151 


\bibitem[CS16d]{cornejo2016derived}
Cornejo JM, Sankappanavar HP (2016). On derived algebras and subvarieties of implication zroupoids.  Soft Comput. Pages 1 - 20.


\bibitem[CS16e]{cornejo2016BolMoufang}
Cornejo JM, Sankappanavar HP (2017). Symmetric implication zroupoids and the identities of Bol-Moufang type,  Soft Comput. 2017 (To appear).

\bibitem[CS16f]{cornejo2016weakAssociative}
Cornejo JM, Sankappanavar HP (2016). Symmetric implication zroupoids and the weak associative identities  (submitted).

\bibitem[CS16g]{cornejo2016varieties}
Cornejo JM, Sankappanavar HP (2016). Varieties of implication zroupoids 
(In preparation).


\bibitem[Fe69]{Fe69} Fenyves F (1969). Extra loops II. On loops with identities of Bol-Moufang type. Publ. Math. Debrecen 16, 187-192 .



\bibitem[Ho54]{Ho54} Hossz M (1954). Some functional equations related with the associative law, Publ.
Math. Debrecen 3, 205-214.

\bibitem[KaNa72]{KaNa72} 
Kazim MA, Naseeruddin M (1972). On almost-semigroups, The Alig Bull Math 2, 1-7.

\bibitem[Ku96]{Ku96}Kunen K (1996). Quasigroups, loops, and associative laws. J. Algebra 185, 194-204 . 

	
\bibitem[Mc] {Mc} McCune W (2005-2010). Prover9 and Mace4.
 \url{http://www.cs.unm.edu/mccune/prover9/ }.
 
\bibitem[Mu41]{Mu41} Murdoch D C (1941). Structure of abelian quasi-groups, Trans. Amer. Math. Soc. 49, 392-409.
 

\bibitem[Pe1880]{Pe1880} Peirce CS (1880). On the algebra of logic, Amer. J. Math. 3, p. 15-57.

\bibitem[Pflu20000]{Pfl2000} Pflugfelder HO (2000). Historical notes on loop theory,  Comment. Math. Univ. Carolin. 41,2 359-370.
  	
\bibitem[PV05a] {PV05a} Phillips JD, Vojtechovsky P (2005). The varieties of loops of Bol-Moufang type, Algebra univers. 54, 259-271.  


\bibitem[PV05b] {PV05b} Phillips JD, Vojtechovsky P (2005). The varieties of quasigroups of Bol-Moufang type: an equational reasoning approach. J. Algebra 293, 17-33.







\bibitem[PrSt04]{PrSt04}
Protic PV, Stevanovic N (2004). Abel-Grassmann bands, Quasi- groups and Relat. Syst., 11,  95-101.

\bibitem[Pu10]{Pu10} Pushkashu DI (2010). Para-associative groupoids, Quasigroups and Related Systems 18, 187-194.

 \bibitem[Ras74]{rasiowa1974algebraic}
Rasiowa H (1974).
\newblock  An algebraic approach to non-classical logics.
\newblock North-Holland Publishing Co., Amsterdam.
\newblock Studies in Logic and the Foundations of Mathematics, Vol. 78.

\bibitem[San12]{sankappanavarMorgan2012} Sankappanavar HP (2012). De {M}organ algebras: new perspectives and applications. Sci. Math. Jpn.  75(1):21--50.

\bibitem[Sc1887]{Sch1887} Schroeder E (1887). Uber algorithmen und Calculi, Arch. der Math. und Phys. 5, 225-278.










\bibitem[St15]{St15} Stanovsky D (2015). A guide to self-distributive quasigroups, or latin quandles,
Quasigroups and Related Systems 23, 91-128.

\bibitem[St1959]{St1959} Stein SK (1959). Left-distributive quasigroups, Proc. Amer. Math. Soc. 10, 577-578.

\bibitem[St1957]{St1957} Stein SK (1957). On the foundations of quasigroups, Trans Amer. Math Soc. 85, 228-256.

\bibitem[Su1929]{Su1929}
Suschkewitsch A (1929). On a generalization of the associative law, Trans. Amer. Math. Soc. 31, no. 1, 204-214.  

\end{thebibliography}
\end{document}